\documentclass{aims}

\usepackage{txfonts}
\def\typeofarticle{Research article}
\def\currentvolume{8}
\def\currentissue{12}
\def\currentyear{2023}

\def\ppages{xxx--xxx.}
\def\DOI{ 10.3934/math.2023xxx}
\def\Received{19 August 2023}
\def\Revised{27 October 2023}
\def\Accepted{02 November 2023}
\def\Published{XX November 2023}
\usepackage{cite}

\numberwithin{equation}{section}

\makeatletter
\renewcommand{\@biblabel}[1]{#1\hfill \hspace{-0.2cm}}
\makeatother

\newcommand{\R}{\mbox{${\rm{I\!R}}$}}
\newcommand{\Q}{\mbox{${\rm{I\!P}}$}}
\newtheorem{theorem}{Theorem}

\newtheorem{proposition}[theorem]{Proposition}
\newtheorem{definition}{Definition}

\begin{document}

\title{Total positivity and dependence of order statistics}

\author{%
  Enrique de Amo\affil{1},
  Jos\'e Juan Quesada-Molina\affil{2}
  and
  Manuel \'Ubeda-Flores\affil{1,}\corrauth
}

\shortauthors{the Author(s)}

\address{%
  \addr{\affilnum{1}}{Department of Mathematics, University of Almer\'ia, 04120 Almer\'ia, Spain}
  \addr{\affilnum{2}}{Department of Applied Mathematics, University of Granada, 18071 Granada, Spain}}

\corraddr{Email: mubeda@ual.es; Tel: +34950214722.\\
}

\begin{abstract}
In this comprehensive study, we delve deeply into the concept of multivariate total positivity, defining it in accordance with a direction. We rigorously explore numerous salient properties, shedding light on the nuances that characterize this notion. Furthermore, our research extends to establishing distinct forms of dependence among the order statistics of a sample from a distribution function. Our analysis aims to provide a nuanced understanding of the interrelationships within multivariate total positivity and its implications for statistical analysis and probability theory.
\end{abstract}


\maketitle

\section{Introduction}\label{sec:intro}
There are different ways to discuss dependence relations among random variables and, as Jogdeo \cite{Jogdeo1982} notes: ``...this is one of the most widely studied objects in probability and statistics."

Recent literature extensively studied the concept of dependence in bivariate and multivariate settings. These concepts are particularly relevant in fields such as economics, insurance, finance, risk management, applied probability and statistics (see, e.g., \cite{Cizek2005}). Several definitions of positive dependence have been introduced to model the association between large values of a component of a multivariate random vector and large values of the other components ---further discussion of most of the dependence concepts that we present in this paper can be consulted in \cite{Barlow1981,Joe1997,Karlin1968,Lee1985,Lehmann1966,Shaked1977,Tong1980}--- including multivariate total positivity of order 2 (MTP$_2$) ---also known as positive likelihood ratio dependence for the bivariate case. This concept has garnered significant attention, particularly in Gaussian models, owing to its intuitive description that highlights the non-negativity of all correlations and partial correlations. In finance, psychometrics, machine learning, medical statistics and phylogenetics, MTP$_2$ models have been shown to outperform state-of-the-art methods; moreover, there is a fundamental connection between the MTP$_2$ constraint and the assumption of sparsity ---see, e.g., \cite{Karlin1983,Slawski2015}.

However, not all dependence concepts, especially in the multivariate case, encompass all dependencies between random variables. In particular, the above mentioned MTP$_2$ concept is defined for the case when the random vector is $(X_1,X_2,\ldots,X_n)$, but not when at least one of the variables is negative: For instance, a random vector of type $(-X_1,X_2,\ldots,X_n)$. Thus, we intend to extend the concept to random variables that follow, for example, the guidelines of the random vector presented.

On the other hand, the $i$-th order statistic of a sample from a distribution functions is equal to its $i$-th smallest value. Together with rank statistics, order statistics are among the most fundamental tools in non-parametric statistics and inference \cite{David2003}. We establish certain types of dependence ---both for some of those previously defined and some well-known that we will recall--- for order statistics.

The paper is organized as follows. In Section \ref{sec:pre}, we provide the major definitions. In Section \ref{sec:direction}, we concentrate on the notion of multivariate total positivity according to a direction and provide several properties. In Section \ref{sec:order}, we establish certain types of dependence for order statistics. Finally, Section \ref{sec:con} is devoted to conclusions.

\section{Preliminaries}\label{sec:pre}
Let $d\ge 2$ be a natural number. Let $(\Omega,\cal{F},\Q)$  be a probability space, where $\Omega$ is a nonempty set, $\cal{F}$ is a $\sigma$-algebra of subsets of $\Omega$, and $\Q$ is a probability measure on $\cal{F}$, and let ${\bf X}=(X_1,\ldots,X_d)$ be a vector of independent and identically distributed (i.i.d. for short) random variables from $\Omega$ to $\overline{\R}^d=[-\infty,\infty]^d$.

A function $f\colon \overline{\mathbb{R}}^2\longrightarrow [0,+\infty[$ is {\it totally positive of order two} ---denoted by TP$_2$--- if
$$f(x',y')f(x,y)\ge f(x',y)f(x,y')$$
whenever $x\le x'$ and $y\le y'$ \cite{Karlin1968}.

Two random variables $X$ and $Y$ are said to be totally positive of order two ---or {\it positively likelihood ratio dependent}, denoted by ${\rm PLRD}(X,Y)$--- if the density function of the pair $\left(X,Y\right)$ is TP$_2$.

In the multivariate case, a generalization of total positivity of order two can be defined \cite{Karlin1980}. A function $f\colon \overline{\mathbb{R}}^d\longrightarrow [0,+\infty[$ is {\it multivariate totally positive of order two} ---denoted by MTP$_2$--- if
$$f({\bf x}\vee{\bf y})f({\bf x}\wedge{\bf y})\ge f({\bf x})f({\bf y})$$
for all ${\bf x}=\left(x_1,\ldots,x_d\right),{\bf y}=\left(y_1,\ldots,y_d\right)\in \overline{\mathbb{R}}^d$, where
\begin{eqnarray*}
{\bf x}\vee{\bf y}&=&\left(\max\left(x_1,y_1\right),\ldots,\max\left(x_d,y_d\right)\right),\\
{\bf x}\wedge{\bf y}&=&\left(\min\left(x_1,y_1\right),\ldots,\min\left(x_d,y_d\right)\right).
\end{eqnarray*}

A random vector ${\bf X}=\left(X_1,\ldots,X_d\right)$ is said to be multivariate totally positive of order two ---or {\it multivariate positively likelihood ratio dependent}--- if its joint $d$-dimensional density $f$ is MTP$_2$. See \cite{Marshall1990} for examples and \cite{Karlin1981,Propp1996} for applications.

We note that by reversing the sense of the above inequalities, we have the corresponding negative concepts, obtaining similar results which we omit here.

In the next sections, when we talk about these ---or other--- dependence concepts, we will refer to random variables or to their joint density functions.

\section{Multivariate total positivity according to a direction}\label{sec:direction}
In this section, we provide a simple generalization of the TP$_2$ concept in higher dimensions in a directional sense, in which a pair of the components of the random vector can be negative. After giving some simple properties of this concept, we provide a ``natural'' generalization of the MTP$_2$ concept according to a direction ---i.e., the components of the random vector can take negative values--- and show that the two newly defined concepts are equivalent. Additional characterizations and properties are given throughout the section.

The next definition generalizes the concept of TP$_2$ to $d$-dimensions according to a direction.

\begin{definition}\label{def:pairs}Let ${\bf X}$ be a $d$-dimensional random vector with joint density $f$, and let $\alpha=\left(\alpha_1,\ldots,\alpha_d\right)\in\mathbb{R}^d$ such that $\left|\alpha_i\right|=1$ for all $i=1,\ldots,d$. The function $f$ is said to be {\it multivariate totally positive of order two in pairs according to the direction $\alpha$} ---denoted by MTPP$_2(\alpha)$--- if
\begin{eqnarray}\nonumber&&f\left(x_1,\ldots,\alpha_i x_i,\ldots,\alpha_j x_j,\ldots,x_d\right)f\left(x_1,\ldots,\alpha_i x_i',\ldots,\alpha_j x_j',\ldots,x_d\right)\\
&&\quad \ge f\left(x_1,\ldots,\alpha_i x_i',\ldots,\alpha_j x_j,\ldots,x_d\right)f\left(x_1,\ldots,\alpha_i x_i,\ldots,\alpha_j x_j',\ldots,x_d\right)\label{eq:pairs}
\end{eqnarray}
for all $\left(x_1,\ldots,x_d,x_i',x_j'\right)\in{\overline{\mathbb{R}}}^{d+2}$ such that $x_i\le x_i'$ and $x_j\le x_j'$ and any election of $(i,j)$.
\end{definition}

Note that if a random pair $(X_1,X_2)$ is TP$_2$ then it is MTPP$_2(1,1)$.

The dependency in MTPP$_2(\alpha)$, whichever is $\alpha$, is positive since the fact that a $d$-dimensional random vector ${\bf X}$ is MTPP$_2(\alpha)$ indicates that large values of the random variables $X_j$, with $j\in J$, correspond with small values of the variables $X_j$, with $j\in I\backslash J$, where $I=\{1,\ldots,d\}$ and $J=\left\{i\in I: \alpha_i>0\right\}$. In addition, there is also association between small values of the variables $X_j$, with $j\in J$, and large values of $X_j$, with $j\in I\backslash J$, as the following result shows.

\begin{proposition}A $d$-dimensional random vector ${\bf X}$ is MTPP$_2(\alpha)$ if, and only if, it is MTPP$_2(-\alpha)$.
\end{proposition}

\begin{proof}Assume ${\bf X}$ is MTPP$_2(\alpha)$, then \eqref{eq:pairs} holds. By making the changes $x_i=-y_i'$, $x_j=-y_j'$, $x_i'=-y_i$ and $x_j'=-y_j$ ---note that $y_i\le y_i'$ and $y_j\le y_j'$--- we easily obtain that ${\bf X}$ is also MTPP$_2(-\alpha)$.

The converse follows the same steps.
\end{proof}

The proof of the next property concerning the MTPP$_2(\alpha)$ concept ---in which {\bf 1} denotes the vector $(1,1,\ldots,1)\in\mathbb{R}^d$--- is simple, and we omit it.

\begin{proposition}\label{pro:alfa}A $d$-dimensional random vector ${\bf X}$ is MTPP$_2(\alpha)$ if, and only if, $\alpha{\bf X}$ is MTPP$_2({\bf 1})$.
\end{proposition}

In the following definition we provide a generalization of the MTP$_2$-concept, similarly to the generalization defined for the orthant dependence given in \cite{QueUb2012}, and where $\alpha {\bf z}$ will denote the $d$-dimensional vector $\left(\alpha_1 z_1,\ldots,\alpha_d z_d\right)$.

\begin{definition}Let ${\bf X}$ be a $d$-dimensional random vector with joint density function $f$, and $\alpha\in{\overline{\mathbb{R}}}^d$, with $\left|\alpha_i\right|=1$ for all $i=1,\ldots,d$. Then ${\bf X}$ is said to be {\it multivariate totally positive of order two according to the direction $\alpha$} ---denoted by MTP$_2(\alpha)$--- if
$$f(\alpha({\bf x}\vee{\bf y}))f(\alpha({\bf x}\wedge{\bf y)})\ge f(\alpha{\bf x})f(\alpha{\bf y})$$
for all ${\bf x}=\left(x_1,\ldots,x_d\right),{\bf y}=\left(y_1,\ldots,y_d\right)\in \overline{\mathbb{R}}^d$.
\end{definition}

In this section, we have defined two multivariate generalizations of the TP$_2$-concept. Now, we prove that both concepts, MTPP$_2(\alpha)$ and MTP$_2(\alpha)$, are equivalent, as the next result shows.

\begin{theorem}\label{pro:charac}A $d$-dimensional random vector ${\bf X}$ is MTP$_2(\alpha)$ if, and only if, it is MTPP$_2(\alpha)$.
\end{theorem}

\begin{proof}Consider the vectors
$${\bf x}=\left(x_1,\ldots,x_i',\ldots,x_j,\ldots,x_d\right)\quad{\rm and}\quad {\bf y}=\left(x_1,\ldots,x_i,\ldots,x_j',\ldots,x_d\right),$$
with $x_i\le x_i'$ and $x_j\le x_j'$. Then we have
\begin{eqnarray*}
\alpha\left({\bf x}\wedge {\bf y}\right)&=&\left(\alpha_1 x_1,\ldots,\alpha_i x_i,\ldots,\alpha_j x_j,\ldots,\alpha_d x_d\right)\\
\alpha\left({\bf x}\vee {\bf y}\right)&=&\left(\alpha_1 x_1,\ldots,\alpha_i x_i',\ldots,\alpha_j x_j',\ldots,\alpha_d x_d\right).
\end{eqnarray*}
If ${\bf X}$ is MTP$_2(\alpha)$, then we immediately obtain that ${\bf X}$ is MTPP$_2(\alpha)$.

Conversely, for ${\bf x},{\bf y}\in{\overline{\mathbb{R}}}^d$ suppose, without loss of generality, $x_l\ge y_l$ for $l=1,\ldots,r$ and $x_l\le y_l$ for $l=r+1,\ldots,d$. Let $s=d-r$. For each $i$, with $0\le i\le r$, and for each $j$, with $0\le j\le s$, we define the vectors
\begin{eqnarray*}
{\bf z}^{i,j}&:=&(x_1\vee y_1,\ldots,x_i\vee y_i,x_{i+1}\wedge y_{i+1},\ldots,x_r\wedge y_r,x_{r+1}\vee y_{r+1},\ldots,x_{r+j}\vee y_{r+j},\\
&&x_{r+j+1}\wedge y_{r+j+1},\ldots,x_d\wedge y_d)
\end{eqnarray*}
(note ${\bf z}^{r,0}={\bf x}$, ${\bf z}^{0,s}={\bf y}$, ${\bf z}^{0,0}={\bf x}\wedge{\bf y}$ and ${\bf z}^{r,s}={\bf x}\vee{\bf y}$). Then we have ${\bf z}^{i+1,j}\wedge{\bf z}^{i,j+1}={\bf z}^{i,j}$ and ${\bf z}^{i+1,j}\vee{\bf z}^{i,j+1}={\bf z}^{i+1,j+1}$. Since ${\bf X}$ is MTPP$_2(\alpha)$, if $h$ is the joint density function of ${\bf X}$, we obtain
$$h\left(\alpha {\bf z}^{i,j}\right)h\left(\alpha {\bf z}^{i+1,j+1}\right)\ge h\left(\alpha {\bf z}^{i+1,j}\right)h\left(\alpha {\bf z}^{i,j+1}\right)$$
for any $(i,j)$ such that $0\le i\le r-1$ and $0\le j\le s-1$; therefore,
$$\prod_{i=0}^{r-1}\prod_{j=0}^{s-1}h\left(\alpha {\bf z}^{i,j}\right)h\left(\alpha {\bf z}^{i+1,j+1}\right)\ge \prod_{i=0}^{r-1}\prod_{j=0}^{s-1}h\left(\alpha {\bf z}^{i+1,j}\right)h\left(\alpha {\bf z}^{i,j+1}\right),$$
whence
$$h\left(\alpha{\bf z}^{0,0}\right)h\left(\alpha{\bf z}^{r,s}\right)\ge h\left(\alpha{\bf z}^{r,0}\right)h\left(\alpha{\bf z}^{0,s}\right),$$
i.e., ${\bf X}$ is MTP$_2(\alpha)$, which completes the proof.
\end{proof}

Another different characterization of the MTP$_2(\alpha)$-concept is given in the following result.

\begin{proposition}\label{pro:charden}A $d$-dimensional random vector ${\bf X}$ with joint density function $h$ is MTP$_2(\alpha)$ if, and only if, for any pair of vectors ${\bf x},{\bf x}'\in{\overline{\mathbb{R}}}^d$ such that $x_i\le x_i'$ for all $i=1,\ldots,d$, and any $1\le j\le d-1$, we have
\begin{equation}\label{eq:desi1}
h\left(\alpha{\bf x}\right)h\left(\alpha{\bf x'}\right)\ge h\left(\alpha{\bf x}'^{j}\right)h\left(\alpha{\bf x}^j\right),
\end{equation}
where ${\bf x}'^{j}=\left(x_1',\ldots,x_j',x_{j+1},\ldots,x_d\right)$ and ${\bf x}^{j}=\left(x_1,\ldots,x_j,x_{j+1}',\ldots,x_d'\right)$.
\end{proposition}

\begin{proof}Assume ${\bf X}$ is MTP$_2(\alpha)$. Let ${\bf x},{\bf x}'\in{\overline{\mathbb{R}}}^d$ such that $x_i\le x_i'$ for all $i=1,\ldots,d$. For any $1\le j\le d-1$ consider the vectors ${\bf x}'^{j}$ and ${\bf x}^j$. Then, it is clear ${\bf x}'^{j}\wedge {\bf x}^j={\bf x}$ and ${\bf x}'^{j}\vee {\bf x}^j={\bf x}'$, whence we obtain \eqref{eq:desi1}.

Conversely, given $i,j\in I=\{1,\ldots,d\}$, let $k=\max\{i,j\}$. Then, for any ${\bf x},{\bf x}'\in{\overline{\mathbb{R}}}^d$ such that $x_l\le x_l'$ for all $l=1,\ldots,d$, from \eqref{eq:desi1} we have
$$h\left(\alpha{\bf x}\right)h\left(\alpha{\bf x'}\right)\ge h\left(\alpha{\bf x}'^{k-1}\right)h\left(\alpha{\bf x}^{k-1}\right).$$
Taking $x_l=x_l'$ for all $l\in I\backslash\{i,j\}$ we easily obtain that ${\bf X}$ is MTPP$_2(\alpha)$, and therefore it is MTP$_2(\alpha)$, completing the proof.
\end{proof}

For the next result, in which we provide another characterization of the MTPP$_2(\alpha)$-concept of a random vector, we need some additional notation: Given a $d$-dimensional random vector ${\bf X}$ with joint distribution function $H$, let $\overline{H}\left(x_1,\ldots,x_d\right)$ denote the probability that ${\bf X}$ is greater than ${\bf x}$ ---also known as the {\it joint survival function of} $H$---, i.e.,
$${\overline{H}}\left(x_1,\ldots,x_d\right)=\Q\left[\bigcap_{i=1}^{d}\left( X_i>x_i\right)\right].$$

\begin{proposition}A $d$-dimensional random vector ${\bf X}$ is MTPP$_2(\alpha)$ if, and only if, ${\overline{H}}$ is MTPP$_2({\bf 1})$.
\end{proposition}

\begin{proof}Assume ${\bf X}$ is MTPP$_2(\alpha)$. Since
$${\overline{H}}\left(x_1,\ldots,x_d\right)=\Q\left[\bigcap_{i=1}^{d}\left( \alpha_i X_i>x_i\right)\right],$$
we have the following chain of equalities:
\begin{eqnarray}\nonumber
&&\Q\left[\bigcap_{i=1}^{d}\left( \alpha_i X_i>x_i\right)\right]\Q\left[\bigcap_{i=1}^{d}\left( \alpha_i X_i>x_i'\right)\right]\\\nonumber
&&\quad -
\Q\left[\bigcap_{i=1}^{j}\left( \alpha_i X_i>x_i\right),\bigcap_{i=j+1}^{d}\left( \alpha_i X_i>x_i'\right)\right]\Q\left[\bigcap_{i=1}^{j}\left( \alpha_i X_i>x_i'\right),\bigcap_{i=j+1}^{d}\left( \alpha_i X_i>x_i\right)\right]\\\nonumber
&&=\Q\left[\bigcap_{i=1}^{j}\left( x_i< \alpha_i X_i\le x_i'\right),\bigcap_{i=j+1}^{d}\left(\alpha_i X_i> x_i\right)\right]\Q\left[\bigcap_{i=1}^{d}\left( \alpha_i X_i>x_i'\right)\right]\\\nonumber
&&\quad -
\Q\left[\bigcap_{i=1}^{j}\left(x_i< \alpha_i X_i\le x_i'\right),\bigcap_{i=j+1}^{d}\left( \alpha_i X_i>x_i\right)\right]\Q\left[\bigcap_{i=1}^{j}\left( \alpha_i X_i>x_i'\right),\bigcap_{i=j+1}^{d}\left( \alpha_i X_i>x_i\right)\right]\\\nonumber
&&=\Q\left[\bigcap_{i=1}^{d}\left( x_i< \alpha_i X_i\le x_i'\right)\right]\Q\left[\bigcap_{i=1}^{d}\left( \alpha_i X_i>x_i'\right)\right]\\
&&\quad -
\Q\left[\bigcap_{i=1}^{j}\left(x_i< \alpha_i X_i\le x_i'\right),\bigcap_{i=j+1}^{d}\left( \alpha_i X_i>x_i\right)\right]\Q\left[\bigcap_{i=1}^{j}\left( \alpha_i X_i>x_i'\right),\bigcap_{i=j+1}^{d}\left(x_i< \alpha_i X_i\le x_i'\right)\right].\label{eqnarray:desi}
\end{eqnarray}

Now, we prove that the last expression in \eqref{eqnarray:desi} is non-negative. For that, note that, from Proposition \ref{pro:alfa}, we have that the random vector $\alpha{\bf X}$ is MTPP$_2({\bf 1})$, and from Proposition \ref{pro:charden} we have
\begin{equation}\label{eq:gden}
g\left({\bf y}\right)g\left({\bf y'}\right)- g\left({\bf y}'^{j}\right)g\left({\bf y}^j\right)\ge 0,
\end{equation}
where $g$ is the joint density function of $\alpha{\bf X}$. Integrating in both sides of \eqref{eq:gden}, with $x_i<y_i\le x_i'<y_i'$ for all $i=1,\ldots,d$, we have
$$\int_{x_1}^{x_1'}\cdots\int_{x_d}^{x_d'}\int_{x_1'}^{\infty}\cdots\int_{x_d'}^{\infty}\left(g\left({\bf y}\right)g\left({\bf y'}\right)- g\left({\bf y}'^{j}\right)g\left({\bf y}^j\right)\right){\rm d}{\bf y}{\rm d}{\bf y}'\ge 0,$$
so that the expression in \eqref{eqnarray:desi} is non-negative.

The converse follows the same steps, and the proof is completed.
\end{proof}

Next we show that any subset of a MTPP$_2(\alpha)$ random vector preserves this property.

\begin{proposition}If ${\bf X}=\left(X_1,\ldots,X_d\right)$ is a MTPP$_2(\alpha)$ random vector, then any subset $\left(X_{i_1},\ldots,X_{i_k}\right)$ of ${\bf X}$ is MTPP$_2(\alpha^*)$, where $\alpha^*=\left(\alpha_{i_1},\ldots,\alpha_{i_k}\right)$.
\end{proposition}

\begin{proof}Let $d\ge 3$ be a natural number, $i\in\{1,\ldots,d\}$, $\alpha^{(i)}=\left(\alpha_1,\ldots,\alpha_{i-1},\alpha_{i+1},\ldots,\alpha_d\right)$ and ${\bf X}^{(i)}$ the $(d-1)$-dimensional random vector $\left(X_1,\ldots,X_{i-1},X_{i+1},\ldots,X_d\right)$. Let $g(x_1,\ldots,x_d)$ (respectively, $g^{(i)}(x_1,\ldots,x_{i-1},x_{i+1},\ldots,x_d)$) be the joint density function of the random vector $\alpha{\bf X}$ (respectively, $\alpha^{(i)}{\bf X}^{(i)}$). We now prove that the random vector $\alpha^{(i)}{\bf X}^{(i)}$ is MTPP$_2\left({\bf 1}_{d-1}\right)$, where ${\bf 1}_{d-1}$ denotes the unitary $d-1$-vector, and, from Proposition \ref{pro:alfa}, we would have that ${\bf X}^{(i)}$ is MTPP$_2\left(\alpha^{(i)}\right)$. Continuing this reasoning for a determined number of components, the result would be proved.

Let $j,k\in\{1,\ldots,d\}$ such that $j< i< k$. For the sake of simplicity, let us denote
$$g^{(i)}\left(x_j,x_k,x^{(j,k)}\right):=g^{(i)}\left(x_1,\ldots,x_j,\ldots,x_{i-1},x_{i+1},\ldots,x_k,\ldots,x_d\right)$$
and
$$g\left(x_j,x_i,x_k,x^{(j,i,k)}\right):=g\left(x_1,\ldots,x_j,\ldots,x_{i},\ldots,x_k,\ldots,x_d\right).$$
Therefore, we need to prove
$$g^{(i)}\left(x_j,x_k,x^{(j,k)}\right)g^{(i)}\left(x_j',x_k',x^{(j,k)}\right)\ge g^{(i)}\left(x_j',x_k,x^{(j,k)}\right)g^{(i)}\left(x_j,x_k',x^{(j,k)}\right)$$
for any $x_j,x_k,x_j',x_k'\in{\overline{\mathbb{R}}}$ such that $x_j\le x_j'$ and $x_k\le x_k'$ and any $x^{(j,k)}\in{\overline{\mathbb{R}}}^{d-3}$.

We have
\begin{eqnarray*}
&&g^{(i)}\left(x_j,x_k,x^{(j,k)}\right)g^{(i)}\left(x_j',x_k',x^{(j,k)}\right)- g^{(i)}\left(x_j',x_k,x^{(j,k)}\right)g^{(i)}\left(x_j,x_k',x^{(j,k)}\right)\\
&&\,\,=\int\!\!\int\frac{g\left(x_j',x_i',x_k',x^{(j,i,k)}\right)}{g\left(x_j',x_i',x_k,x^{(j,i,k)}\right)}g\left(x_j,x_i,x_k,x^{(j,i,k)}\right)g\left(x_j',x_i',x_k,x^{(j,i,k)}\right){\rm d}x_i{\rm d}x_i'\\
&&\quad-\int\!\!\int\frac{g\left(x_j,x_i',x_k',x^{(j,i,k)}\right)}{g\left(x_j,x_i',x_k,x^{(j,i,k)}\right)}g\left(x_j',x_i,x_k,x^{(j,i,k)}\right)g\left(x_j,x_i',x_k,x^{(j,i,k)}\right){\rm d}x_i{\rm d}x_i'\\
&&\,\,=\int\!\!\int_{x_i<x_i'}\left(\frac{g\left(x_j',x_i',x_k',x^{(j,i,k)}\right)}{g\left(x_j',x_i',x_k,x^{(j,i,k)}\right)}-
\frac{g\left(x_j,x_i,x_k',x^{(j,i,k)}\right)}{g\left(x_j,x_i,x_k,x^{(j,i,k)}\right)}\right)\\
&&\qquad \cdot\left(g\left(x_j,x_i,x_k,x^{(j,i,k)}\right)g\left(x_j',x_i',x_k,x^{(j,i,k)}\right)-
g\left(x_j',x_i,x_k,x^{(j,i,k)}\right)g\left(x_j,x_i',x_k,x^{(j,i,k)}\right)\right){\rm d}x_i{\rm d}x_i'\\
&&\quad +\int\!\!\int_{x_i<x_i'}\biggl[\frac{g\left(x_j',x_i',x_k',x^{(j,i,k)}\right)}{g\left(x_j',x_i',x_k,x^{(j,i,k)}\right)}-
\frac{g\left(x_j,x_i',x_k',x^{(j,i,k)}\right)}{g\left(x_j,x_i',x_k,x^{(j,i,k)}\right)}+
\frac{g\left(x_j',x_i,x_k',x^{(j,i,k)}\right)}{g\left(x_j',x_i,x_k,x^{(j,i,k)}\right)}\\
&&\qquad\qquad-\frac{g\left(x_j,x_i,x_k',x^{(j,i,k)}\right)}{g\left(x_j,x_i,x_k,x^{(j,i,k)}\right)}\biggr]\cdot g\left(x_j',x_i,x_k,x^{(j,i,k)}\right)g\left(x_j,x_i',x_k,x^{(j,i,k)}\right){\rm d}x_i{\rm d}x_i'\ge 0\\
\end{eqnarray*}
since $\alpha{\bf X}$ is MTPP$_2({\bf 1})$ ---recall Proposition \ref{pro:alfa}--- and
$$g\left(x_j,x_i,x_k,x^{(j,i,k)}\right)g\left(x_j',x_i',x_k,x^{(j,i,k)}\right)-
g\left(x_j',x_i,x_k,x^{(j,i,k)}\right)g\left(x_j,x_i',x_k,x^{(j,i,k)}\right)\ge 0,$$
$$\frac{g\left(x_j',x_i',x_k',x^{(j,i,k)}\right)}{g\left(x_j',x_i',x_k,x^{(j,i,k)}\right)}\ge
\frac{g\left(x_j,x_i',x_k',x^{(j,i,k)}\right)}{g\left(x_j,x_i',x_k,x^{(j,i,k)}\right)}\ge
\frac{g\left(x_j,x_i,x_k',x^{(j,i,k)}\right)}{g\left(x_j,x_i,x_k,x^{(j,i,k)}\right)},$$
$$\frac{g\left(x_j',x_i',x_k',x^{(j,i,k)}\right)}{g\left(x_j',x_i',x_k,x^{(j,i,k)}\right)}\ge
\frac{g\left(x_j,x_i',x_k',x^{(j,i,k)}\right)}{g\left(x_j,x_i',x_k,x^{(j,i,k)}\right)}\quad{\rm and}\quad
\frac{g\left(x_j',x_i,x_k',x^{(j,i,k)}\right)}{g\left(x_j',x_i,x_k,x^{(j,i,k)}\right)}\ge
\frac{g\left(x_j,x_i,x_k',x^{(j,i,k)}\right)}{g\left(x_j,x_i,x_k,x^{(j,i,k)}\right)};$$
whence the result follows.
\end{proof}

For the next result, we will use some additional notation. For $\alpha=\left(\alpha_1,\ldots,\alpha_{d_1}\right)\in{\overline{\mathbb{R}}}^{d_1}$ and
$\beta=\left(\beta_1,\ldots,\beta_{d_2}\right)\in{\overline{\mathbb{R}}}^{d_2}$, $(\alpha,\beta)$ will denote {\it concatenation}, i.e.,
$(\alpha,\beta)=\left(\alpha_1,\ldots,\alpha_{d_1},\beta_1,\ldots,\beta_{d_2}\right)\in{\overline{\mathbb{R}}}^{d_1+d_2}$; and similarly for random vectors.

\begin{proposition}If the respective $d_1$- and $d_2$-dimensional random vectors ${\bf X}$ and ${\bf Y}$ are MTPP$_2(\alpha)$ and MTPP$_2(\beta)$, and ${\bf X}$ and ${\bf Y}$ are independent, then the $\left(d_1+d_2\right)$-random vector $({\bf X},{\bf Y})$ is MTPP$_2(\alpha,\beta)$.
\end{proposition}

\begin{proof}Since the random vectors ${\bf X}$ and ${\bf Y}$ are independent, so are the variables $\alpha{\bf X}$ and $\beta{\bf Y}$. Denoting by $f({\bf x})$, $g({\bf y})$ and $h({\bf x},{\bf y})$ the respective joint density functions of $\alpha{\bf X}$, $\beta{\bf Y}$ and $(\alpha{\bf X},\beta{\bf Y})$, we have $h({\bf x},{\bf y})=f({\bf x})g({\bf y})$; whence,
\begin{eqnarray*}
&&h({\bf x},{\bf y})h\left(x_1,\ldots,x_i',\ldots,x_{d_1},y_1,\ldots,y_j',\ldots,y_{d_2}\right)\\
&&\quad -h\left(x_1,\ldots,x_i',\ldots,x_{d_1},y_1,\ldots,y_j,\ldots,y_{d_2}\right)h\left(x_1,\ldots,x_i,\ldots,x_{d_1},y_1,\ldots,y_j',\ldots,y_{d_2}\right)=0
\end{eqnarray*}
for any choice $(i,j)$, with $1\le i\le d_1$ and $1\le j\le d_2$, such that $x_i\le x_i'$ and $y_j\le y_j'$. If the two indices chosen are from the first $d_1$ indices (we have a similar reasoning for the last $d_2$ indices), since ${\bf X}$ is  MTPP$_2(\alpha)$, or equivalently, $\alpha{\bf X}$ is  MTPP$_2({\bf 1})$ ---recall Proposition \ref{pro:alfa}---, we have
\begin{eqnarray*}
&&h({\bf x},{\bf y})h\left(x_1,\ldots,x_i',\ldots,x_j',\ldots,x_{d_1},y_1,\ldots,y_{d_2}\right)\\
&&\quad -h\left(x_1,\ldots,x_i',\ldots,x_j,\ldots,x_{d_1},y_1,\ldots,y_{d_2}\right)h\left(x_1,\ldots,x_i,\ldots,x_j',\ldots,x_{d_1},y_1,\ldots,y_{d_2}\right)\\
&&=g({\bf y})\biggl[f({\bf x})f\left(x_1,\ldots,x_i',\ldots,x_j',\ldots,x_{d_1}\right)\\
&&\quad -f\left(x_1,\ldots,x_i',\ldots,x_j,\ldots,x_{d_1}\right)f\left(x_1,\ldots,x_i,\ldots,x_j',\ldots,x_{d_1}\right)\biggr]\ge 0;
\end{eqnarray*}
therefore, $({\bf X},{\bf Y})$ is MTPP$_2(\alpha,\beta)$, which completes the proof.
\end{proof}

We conclude this section with three additional properties of the MTPP$_2(\alpha)$-concept. The first property is straightforward and we omit its proof.

\begin{proposition}Every independent $d$-dimensional random vector ${\bf X}$ is MTPP$_2(\alpha)$ for any $\alpha\in{\overline{\mathbb{R}}}^d$.
\end{proposition}

\begin{proposition}If the random vector ${\bf X}=\left(X_1,\ldots,X_d\right)$ is MTPP$_2(\alpha)$ and $f_1,\ldots,f_d$ are $d$ real-valued and non-decreasing functions, then the random vector $\left(f_1\left({\bf X}_1\right),\ldots,f_d\left({\bf X}_d\right)\right)$ is MTPP$_2(\alpha)$.
\end{proposition}

\begin{proof}Let $f$ (respectively, $g$) be the joint density function of $\alpha{\bf X}$ (respectively, $\left(\alpha_1f_1\left({\bf X}_1\right),\ldots,\alpha_d f_d\left({\bf X}_d\right)\right)$). Then it holds
$$g\left(x_1\ldots,x_d\right)=f\left(\alpha_1f_1^{-1}\left(\alpha_1 x_1\right),\ldots,\alpha_d f_d^{-1}\left(\alpha_d x_d\right)\right).$$
Since ${\bf X}$ is MTPP$_2(\alpha)$, then $\alpha{\bf X}$ is MTPP$_2({\bf 1})$ ---recall Proposition \ref{pro:alfa}---, and since $\alpha_kf_k^{-1}\left(\alpha_k x_k\right)\le \alpha_kf_k^{-1}\left(\alpha_k x_k'\right)$ for $k=i,j$ with $x_k\le x_k'$, then $\left(\alpha_1f_1\left({\bf X}_1\right),\ldots,\alpha_d f_d\left({\bf X}_d\right)\right)$ is MTPP$_2({\bf 1})$, and thus $\left(f_1\left({\bf X}_1\right),\ldots,f_d\left({\bf X}_d\right)\right)$ is MTPP$_2(\alpha)$.
\end{proof}

\begin{proposition}Let $X_1,\ldots,X_d,Y$ be $d+1$ random variables such that $X_1\ldots,X_d$ are independent given $Y$. If the random pair $\left(X_i,Y\right)$ is MTPP$_2\left(\alpha_i,1\right)$ with $\left|\alpha_i\right|=1$ for all $i=1\ldots,d$, then the random vector $\left(X_1,\ldots,X_d\right)$ is MTPP$_2\left(\alpha_1,\ldots,\alpha_d\right)$.
\end{proposition}

\begin{proof}Let $f_i\left(x_i,y\right)$ the joint density function of the random pair $\left(X_i,Y\right)$ for $i=1,\ldots,d$, and let $g(y)$ the density function of $Y$. Then, the joint density function of the random vector $\left(X_1,\ldots,X_d\right)$, which we denote by $f$, is given by
$$f\left(x_1,\ldots,x_d\right)=\int\prod_{i=1}^{d}f_i\left(x_i,y\right)g(y){\rm d}y.$$
Given $i,j\in\{1,\ldots,d\}$ and $x_1,\ldots,x_d,x_i',x_j'\in{\overline{\mathbb{R}}}^d$ with $x_i\le x_i'$ and $x_j\le x_j'$, we have
\begin{eqnarray*}
&&f\left(x_1,\ldots,\alpha_ix_i,\ldots,\alpha_jx_j,\ldots,x_d\right)f\left(x_1,\ldots,\alpha_ix_i',\ldots,\alpha_jx_j',\ldots,x_d\right)\\
&&\quad -f\left(x_1,\ldots,\alpha_ix_i',\ldots,\alpha_jx_j,\ldots,x_d\right)f\left(x_1,\ldots,\alpha_ix_i,\ldots,\alpha_jx_j',\ldots,x_d\right)\\
&&=\int\!\!\int\prod_{\substack{k=1\\i\neq k\neq j}}^{d}\left(f_k\left(x_k,y\right)f_k\left(x_k,y'\right)\right)f_i\left(\alpha_i x_i,y\right)
f_i\left(\alpha_i x_i',y'\right)f_j\left(\alpha_j x_j,y\right)f_j\left(\alpha_j x_j',y'\right)g(y)g\left(y'\right){\rm d}y{\rm d}y'\\
&&\quad-\int\!\!\int\prod_{\substack{k=1\\i\neq k\neq j}}^{d}\left(f_k\left(x_k,y\right)f_k\left(x_k,y'\right)\right)f_i\left(\alpha_i x_i',y\right)
f_i\left(\alpha_i x_i,y'\right)f_j\left(\alpha_j x_j,y\right)f_j\left(\alpha_j x_j',y'\right)g(y)g\left(y'\right){\rm d}y{\rm d}y'\\
&&=\int\!\!\int_{y<y'}\prod_{\substack{k=1\\i\neq k\neq j}}^{d}\left(f_k\left(x_k,y\right)f_k\left(x_k,y'\right)\right)f_i\left(\alpha_i x_i,y\right)
f_i\left(\alpha_i x_i',y'\right)f_j\left(\alpha_j x_j,y\right)f_j\left(\alpha_j x_j',y'\right)g(y)g\left(y'\right){\rm d}y{\rm d}y'\\
&&\quad +\int\!\!\int_{y>y'}\prod_{\substack{k=1\\i\neq k\neq j}}^{d}\left(f_k\left(x_k,y\right)f_k\left(x_k,y'\right)\right)f_i\left(\alpha_i x_i,y\right)
f_i\left(\alpha_i x_i',y'\right)f_j\left(\alpha_j x_j,y\right)f_j\left(\alpha_j x_j',y'\right)g(y)g\left(y'\right){\rm d}y{\rm d}y'\\
&&\quad -\int\!\!\int_{y<y'}\prod_{\substack{k=1\\i\neq k\neq j}}^{d}\left(f_k\left(x_k,y\right)f_k\left(x_k,y'\right)\right)f_i\left(\alpha_i x_i',y\right)
f_i\left(\alpha_i x_i,y'\right)f_j\left(\alpha_j x_j,y\right)f_j\left(\alpha_j x_j',y'\right)g(y)g\left(y'\right){\rm d}y{\rm d}y'\\
&&\quad -\int\!\!\int_{y>y'}\prod_{\substack{k=1\\i\neq k\neq j}}^{d}\left(f_k\left(x_k,y\right)f_k\left(x_k,y'\right)\right)f_i\left(\alpha_i x_i',y\right)
f_i\left(\alpha_i x_i,y'\right)f_j\left(\alpha_j x_j,y\right)f_j\left(\alpha_j x_j',y'\right)g(y)g\left(y'\right){\rm d}y{\rm d}y'\\
&&=\int\!\!\int_{y<y'}\prod_{\substack{k=1\\i\neq k\neq j}}^{d}\left(f_k\left(x_k,y\right)f_k\left(x_k,y'\right)\right)\biggl(f_i\left(\alpha_i x_i,y\right)
f_i\left(\alpha_i x_i',y'\right)f_j\left(\alpha_j x_j,y\right)f_j\left(\alpha_j x_j',y'\right)\\
&&\quad +f_i\left(\alpha_i x_i,y'\right)f_i\left(\alpha_i x_i',y\right)f_j\left(\alpha_j x_j,y'\right)f_j\left(\alpha_j x_j',y\right)-
f_i\left(\alpha_i x_i',y\right)f_i\left(\alpha_i x_i,y'\right)f_j\left(\alpha_j x_j,y\right)f_j\left(\alpha_j x_j',y'\right)\\
&&\quad -f_i\left(\alpha_i x_i',y'\right)f_i\left(\alpha_i x_i,y\right)f_j\left(\alpha_j x_j,y'\right)f_j\left(\alpha_j x_j',y\right)\biggr)g(y)g\left(y'\right){\rm d}y{\rm d}y'\\
&&=\int\!\!\int_{y<y'}\prod_{\substack{k=1\\i\neq k\neq j}}^{d}\left(f_k\left(x_k,y\right)f_k\left(x_k,y'\right)\right)\left(f_i\left(\alpha_i x_i,y\right)
f_i\left(\alpha_i x_i',y'\right)-f_i\left(\alpha_i x_i',y\right)f_i\left(\alpha_i x_i,y'\right)\right)\\
&&\quad\cdot\left(f_j\left(\alpha_j x_j,y\right)f_j\left(\alpha_j x_j',y'\right)-f_j\left(\alpha_j x_j',y\right)f_j\left(\alpha_j x_j,y'\right)\right)g(y)g\left(y'\right){\rm d}y{\rm d}y'\ge 0
\end{eqnarray*}
since $\left(X_l,Y\right)$ is MTPP$_2\left(\alpha_l,1\right)$ for $l=i,j$; therefore, ${\bf X}$ is MTPP$_2(\alpha)$, which completes the proof.
\end{proof}

\section{Dependence of order statistics}\label{sec:order}
In this section, we establish certain types of dependence ---both for some of those previously defined and some known ones that we recall in the section--- for order statistics. We begin by recalling some concepts related to order statistics. We refer to \cite{Arnold2008,David2003} and the references therein for an overview.

Let $X_1, X_2,\ldots, X_d$ be independent and identically distributed random variables, with density function $f$ and distribution function $F$, and $X_{(i)}$ denotes the $i$-th order statistic, being $X_{(1)}=\min\left\{X_1,X_2,\ldots,X_d\right\}$ and $X_{(d)}=\max\left\{X_1,X_2,\ldots,X_d\right\}$. Observe that, for $1\le i< j\le d$, the joint density function of the random pair $\left(X_{(i)},X_{(j)}\right)$ is given by
\begin{equation}\label{eq:density}
h_{i,j}\left(x_i,x_j\right)=\left\{\begin{array}{lll} \displaystyle
m_{i,j}[F(x_i)]^{i-1}[F(x_j)-F(x_i)]^{j-i-1}[1-F(x_j)]^{d-j}f(x_i)f(x_j), & \mbox{if\, $x_i\leq x_j$,}\\ \noalign{\smallskip}
0, &\mbox{otherwise,}
\end{array} \right.
\end{equation}
where
$$m_{i,j}=\frac{d!}{(i-1)!(j-i-1)!(d-j)!},$$
and the joint density function of the random vector $\left(X_{(1)},X_{(2)},\ldots,X_{(i)}\right)$, with $2\le i\le d$, is given by
\begin{equation}\label{eq:density2}
h_{1,2,\ldots,i}\left(x_1,x_2,\ldots,x_i\right)=\left\{\begin{array}{lll} \displaystyle
\frac{d!}{(d-i)!}\left[1-F\left(x_i\right)\right]^{d-i}\prod_{k=1}^{i}f\left(x_k\right), &\mbox{if\, $x_1\le x_2\le\cdots\le x_i$,}\\ \noalign{\smallskip}
0, &\mbox{otherwise.}
\end{array} \right.
\end{equation}

Now we recall several known dependence concepts: see \cite{Colangelo2006,Nappo2020,Wei2014} for more details and applications.

\begin{definition}\label{def:failure}Let $X$ be a random variable with distribution function $F$. Then $F$ is said to be {\it decreasing failure rate} ---denoted by {\rm DFR}--- if $\Q[X>x+y|X>x]$ is a nondecreasing function of $x$ for all $y\ge 0$.
\end{definition}
We note that in Definition \ref{def:failure}, denoting ${\overline{F}}(x)=1-F(x)$, we have that $F$ is decreasing failure rate if $\frac{{\overline{F}}(x+y)}{{\overline{F}}(x)}$ is non-decreasing in $x$ for all $y\ge 0$.

\begin{definition}\label{def:si}(\cite{Colangelo2006}) Let $X$ and $Y$ be random variables. $Y$ is {\it positive regression dependent} in $X$ ---denoted by ${\rm PRD}(Y|X)$--- if $\Q[Y>y|X=x]$ is a nondecreasing function of $x$ for all $y\ge 0$.
\end{definition}

A generalization of the PRD-concept for $d$ random variables is the following.

\begin{definition}The random variables $X_1,X_2,\ldots,X_d$ are {\it conditionally increasing in sequence} ---denoted by {\rm CIS}--- if $\Q\left[X_i>x_i|X_1=x_1,\ldots, X_{i-1}=x_{i-1}\right]$ is nondecreasing in $x_1,\ldots,x_{i-1}$, for every $2\le i\le d$, and for all $x_i\in{\overline{\mathbb{R}}}$.
\end{definition}

In what follows, we provide some results on dependence concepts described in this section and the previous one for order statistics.

\begin{proposition}Let $d$ be a natural number such that $d\ge 2$. For $1\le i\le d$, let $X_{(i)}$ be the $i$-th order statistics of a statistical sample of size $d$ from a {\rm DFR} distribution function $F$, with density function $f$. Then, for every $(i,j)$ such that $1\le i<j\le d$, we have ${\rm PRD}\left(X_{(j)}-X_{(i)}|X_{(i)}\right)$.
\end{proposition}

\begin{proof}We have to prove that $\Q\left[X_{(j)}-X_{(i)}>y|X_{(i)}=x\right]$ is nondecreasing in $x$ for all $y\ge 0$. For that, note
\begin{eqnarray*}
\Q\left[X_{(j)}-X_{(i)}>y|X_{(i)}=x\right]&=&\int_{x+y}^{+\infty}\frac{(d-i)!}{(j-i-1)!(d-j)!}[F(z)-F(x)]^{j-i-1}
\frac{[1-F(z)]^{d-j}}{[1-F(x)]^{d-i}}f(z)\,{\rm d}z\\
&=&\frac{(d-i)!}{(j-i-1)!(d-j)!}\int_{x+y}^{+\infty}\left[1-\frac{{\overline{F}}(z)}{{\overline{F}}(x)}\right]^{j-i-1}
\left[\frac{{\overline{F}}(z)}{{\overline{F}}(x)}\right]^{d-j}\frac{f(z)}{{\overline{F}}(x)}\,{\rm d}z.
\end{eqnarray*}
With the change of variable $u=\frac{{\overline{F}}(z)}{{\overline{F}}(x)}$ we get
$$\Q\left[X_{(j)}-X_{(i)}>y|X_{(i)}=x\right]=\frac{(d-i)!}{(j-i-1)!(d-j)!}\int_{0}^{\frac{{\overline{F}}(x+y)}{{\overline{F}}(x)}}u^{d-j}(1-u)^{j-i-1}\, {\rm d}u=\Q\left[Z\le\frac{{\overline{F}}(x+y)}{{\overline{F}}(x)}\right],$$
where $Z$ is the random variable with Beta distribution $\beta(d-j+1,j-i)$. Since $F$ is {\rm DFR}, then given $x,x'$ in $\mathbb{\mathbb{R}}$, with $x<x'$, we have
$$\frac{{\overline{F}}(x+y)}{{\overline{F}}(x)}\le \frac{{\overline{F}}\left(x'+y\right)}{{\overline{F}}\left(x'\right)}$$
for all $y\ge 0$. Therefore, it is easy to conclude ${\rm PRD}\left(X_{(j)}-X_{(i)}|X_{(i)}\right)$.
\end{proof}

\begin{proposition}For $k=1,2,\ldots,d$, let $X_{(k)}$ be the $k-$th order statistic of a statistical sample of size $d$ from a distribution function $F$. Then we have ${\rm PLRD}\left(X_{(i)},X_{(j)}\right)$ for every $(i,j)$ such that $1\le i,j\le d$.
\end{proposition}

\begin{proof}Assume ---without loss of generality--- $i<j$. Let $h_{i,j}$ be the joint density function of the random vector $\left(X_{(i)},X_{(j)}\right)$ given by \eqref{eq:density}. For any  $x, x', y, y ' \in \overline{\mathbb{R}}$ such that  $x\leq x',  y \leq y'$,  we have
\begin{eqnarray*}
&&h_{i,j}\left(x',y'\right)h_{i,j}(x,y)-h_{i,j}\left(x',y\right)h_{i,j}\left(x,y'\right)=\left(\frac{d!}{(i-1)!(j-i-1)!(d-j)!}\right)^2[F(x)]^{i-1}[F(x')]^{i-1}\\
&&\quad\times [1-F(y)]^{d-j}\left[1-F\left(y'\right)\right]^{d-j}f(x)f\left(x'\right)f(y)f\left(y'\right)\\
&&\quad\times\left\{\left[\left(F(y)-F(x)\right)\left(F\left(y'\right)-F\left(x'\right)\right)\right]^{j-i-1}-
\left[\left(F\left(y'\right)-F(x)\right)\left(F(y)-F\left(x'\right)\right)\right]^{j-i-1}\right\},
\end{eqnarray*}
whence the proof reduces to proving
$$\left(F(y)-F(x)\right)\left(F\left(y'\right)-F\left(x'\right)\right)\ge \left(F\left(y'\right)-F(x)\right)\left(F(y)-F\left(x'\right)\right)$$
which, it reduces in turn to
$$\left(F\left(x'\right)-F(x)\right)\left(F\left(y'\right)-F(y)\right)\ge 0;$$
but this obviously holds since $x\leq x'$ and $y\leq y'$, which completes the proof.
\end{proof}

\begin{proposition} The order statistics $X_{(1)}, X_{(2)},..., X_{(d)}$, of an statistical sample of size $d$ from a distribution function $F$, are always conditionally increasing in sequence (CIS).
\end{proposition}

\begin{proof}Let $2\le i\le d$. Since the joint density function of $\left(X_{(1)},X_{(2)},\ldots,X_{(i)}\right)$ is given by \eqref{eq:density2}, then we have
\begin{eqnarray*}
\Q\left[X_{(i)}>x|\bigcap_{j=1}^{i-1}\left(X_{(j)}=x_j\right)\right]&=&\int_{x}^{+\infty}
\frac{h_{1,2,\ldots,i}\left(x_1,x_2,\ldots,x_i\right)}{h_{1,2,\ldots,i-1}\left(x_1,x_2,\ldots,x_{i-1}\right)}\,{\rm d}x_i\\
&=&\left\{\begin{array}{lll} \displaystyle
\left[\frac{1-F(x)}{1-F\left(x_{i-1}\right)}\right]^{d-i+1}, &\mbox{if\, $x\ge x_{i-1}$,}\\ \noalign{\smallskip}
1, &\mbox{if\, $x<x_{i-1}$.}
\end{array} \right.
\end{eqnarray*}
Therefore, for every $x_{i-1},x_{i-1}'\in{\overline{\mathbb{R}}}$, with $x_{i-1}\le x_{i-1}'$, since $1-F\left(x_{i-1}\right)\ge 1-F\left(x_{i-1}'\right)$, we have
$$\left(\frac{1-F(x)}{1-F\left(x_{i-1}\right)}\right)^{d-i+1}\le \left(\frac{1-F(x)}{1-F\left(x_{i-1}'\right)}\right)^{d-i+1},$$
whence it easily follows the result.
\end{proof}

\begin{proposition}For $1\le i\le d$, let $X_{(i)}$ be the $i$-th order statistic of a statistical sample of size $d$ from a {\rm DFR} distribution $F$. Let $D_1=X_{(1)}$ and $D_i=X_{(i)}-X_{(i-1)}$ for $i=2,\ldots,d$. Then the random variables $D_{(1)}, D_{(2)},..., D_{(d)}$ are CIS.
\end{proposition}

\begin{proof}Given $2\le i\le d$, let $g_i$ denote the joint density function of $\left(D_1,D_2,\ldots,D_i\right)$. Then we have
$$g_i\left(x_1,x_2,\ldots,x_i\right)=h_{1,2\ldots,i}\left(y_1,y_2,\ldots,y_i\right),$$
where $h_{1,2,\ldots,i}$ is the density of $\left(X_{(1)},X_{(2)},\ldots,X_{(i)}\right)$ given by \eqref{eq:density2}, $y_1=x_1$ and $y_j=\sum_{k=1}^{j}x_k$ for $j=2,\ldots,i$. Therefore, for $x\ge 0$ we get
\begin{eqnarray*}
\Q\left[D_{i}>x|\bigcap_{j=1}^{i-1}\left(D_{j}=x_j\right)\right]&=&\int_{x}^{+\infty}
\frac{h_{1,2,\ldots,i}\left(x_1,x_1+x_2,\ldots,\sum_{j=1}^{i}x_j\right)}{h_{1,2,\ldots,i-1}\left(x_1,x_1+x_2,\ldots,\sum_{j=1}^{i-1}x_{j}\right)}\,{\rm d}x_i\\
&=&
\left[\frac{1-F\left(\sum_{j=1}^{i-1}x_j+x\right)}{1-F\left(\sum_{j=1}^{i-1}x_{j}\right)}\right]^{d-i+1},
\end{eqnarray*}
and $\Q\left[D_{i}>x|\bigcap_{j=1}^{i-1}\left(D_{j}=x_j\right)\right]=1$, when $x<0$.\\

Since $F$ is DFR, given $x_1,x_2,\ldots,x_{i-1},x_1',x_2',\ldots,x_{i-1}'\in {\overline{\mathbb{R}}}$ such that $x_j\le x_j'$ for all $j=1,2,\ldots,i-1$, we have $\sum_{j=1}^{i-1}x_j\le\sum_{j=1}^{i-1}x_j'$ and thus
$$\frac{{\overline{F}}\left(\sum_{j=1}^{i-1}x_j+x\right)}{{\overline{F}}\left(\sum_{j=1}^{i-1}x_j\right)}\le \frac{{\overline{F}}\left(\sum_{j=1}^{i-1}x_j'+x\right)}{{\overline{F}}\left(\sum_{j=1}^{i-1}x_j'\right)},$$
and hence
$$\left(\frac{1-F\left(\sum_{j=1}^{i-1}x_j+x\right)}{1-F\left(\sum_{j=1}^{i-1}x_j\right)}\right)^{d-i+1}\le \left(\frac{1-F\left(\sum_{j=1}^{i-1}x_j'+x\right)}{1-F\left(\sum_{j=1}^{i-1}x_j'\right)}\right)^{d-i+1},$$
whence the result follows.
\end{proof}

\section{Conclusions}\label{sec:con}
We have defined two multivariate generalizations of the TP$_2$-concept: Total positivity according to a direction and in pairs. The equivalence of these two multivariate dependence concepts and several key properties have been established. Moreover, specific dependencies among the order statistics of a sample from a distribution function have been identified. Future work will explore additional dependence concepts in multivariate settings.

\section*{Use of AI tools declaration}
The authors declare that they have not used Artificial Intelligence (AI) tools in the creation of this article.

\section*{Acknowledgments}
The authors acknowledge the comments of the reviewers, the support of the program ERDF-Andaluc\'ia 2014-2020 (University of Almer\'ia) under research project UAL2020-AGR-B1783 and project PID2021-122657OB-I00 by the Ministry of Science and Innovation (Spain). The first author is also partially supported by the CDTIME of the University of Almer\'ia.

\section*{Conflict of interest}
The authors declare that they have no known competing financial interests or personal relationships that could have appeared to
influence the work reported in this paper.



\end{document}